\documentclass[12pt]{amsart}
\usepackage{amsmath,amsthm}
\usepackage[a4paper,left=2.5cm,right=2.5cm,top=3.4cm,bottom=3.0cm]{geometry}

\title[Equivalence of recurrence and Liouville property for Dirichlet forms]{Equivalence of recurrence and Liouville property for symmetric Dirichlet forms}

\author{Naotaka Kajino}
\address{Department of Mathematics, Graduate School of Science, Kobe University, Rokkodai-cho 1-1, Nada-ku, Kobe 657-8501, Japan}
\email{nkajino@math.kobe-u.ac.jp}
\thanks{JSPS Research Fellow PD (20$\cdot$6088): The author was supported by the Japan Society for the Promotion of Science.}
\date{October 2, 2017}


\newtheoremstyle{theorem}%name
  {10pt}		  % space above
  {10pt}  % space below
  {\sl}  % bofy font
  {\parindent}     % ident - empty=no indent,  \parindent= paragraph indent
  {\bf}  % thm head font
  {. }    % punctuation after thm head
  { }    % space after thm head: `` ``=normal \newline=linebreak
  {}     % thm head specification
\theoremstyle{theorem}
\newtheorem{thm}{Theorem}
\newtheorem{crl}{Corollary}
\newtheorem{prop}{Proposition}
\newtheorem{lem}{Lemma}

\newtheoremstyle{defi}%name
  {10pt}		  % space above
  {10pt}  % space below
  {\rm}  % bofy font
  {\parindent}     % ident - empty=no indent,  \parindent= paragraph indent
  {\bf}  % thm head font
  {. }    % punctuation after thm head
  { }    % space after thm head: `` ``=normal \newline=linebreak
  {}     % thm head specification
\theoremstyle{defi}
\newtheorem{dfn}{Definition}

\newtheoremstyle{remk}%name
  {10pt}		  % space above
  {10pt}  % space below
  {\rm}  % bofy font
  {\parindent}     % ident - empty=no indent,  \parindent= paragraph indent
  {\sl}  % thm head font
  {. }    % punctuation after thm head
  { }    % space after thm head: `` ``=normal \newline=linebreak
  {}     % thm head specification
\theoremstyle{remk}
\newtheorem{rmk}{Remark}
\newtheorem*{ntn}{Notation}

\def\proofname{\indent {\sl Proof.}}

\usepackage{amssymb,mathrsfs}
\usepackage{bm}
\usepackage[dvips]{hyperref}
\numberwithin{equation}{section}
\newcommand{\meas}{m}
\newcommand{\ind}[1]{\mathbf{1}_{#1}}%\textrm{\rmfamily\bfseries\upshape 1}_{#1}}
\begin{document}

\begin{abstract}
Given a symmetric Dirichlet form
$(\mathcal{E},\mathcal{F})$ on a (non-trivial) $\sigma$-finite measure space
$(E,\mathcal{B},\meas)$ with associated Markovian semigroup
$\{T_{t}\}_{t\in(0,\infty)}$, we prove that $(\mathcal{E},\mathcal{F})$ is
both irreducible and recurrent if and only if there is no non-constant
$\mathcal{B}$-measurable function
$u:E\to[0,\infty]$ that is \emph{$\mathcal{E}$-excessive},
i.e., such that $T_{t}u\leq u$ $\meas$-a.e.\ for any $t\in(0,\infty)$.
We also prove that these conditions are equivalent to the
equality $\{u\in\mathcal{F}_{e}\mid \mathcal{E}(u,u)=0\}=\mathbb{R}\ind{}$,
where $\mathcal{F}_{e}$ denotes the extended Dirichlet space associated with
$(\mathcal{E},\mathcal{F})$. The proof is based on simple analytic arguments
and requires no additional assumption on the state space or on the form.
In the course of the proof we also present a characterization of the
$\mathcal{E}$-excessiveness in terms of $\mathcal{F}_{e}$ and $\mathcal{E}$,
which is valid for any symmetric positivity preserving form.
\end{abstract}

\maketitle
\section{Introduction and the statement of the main theorem}\label{sec:intro}
Since the classical theorem of Liouville saying that there is no non-constant
bounded holomorphic function on $\mathbb{C}$, non-existence of non-constant bounded
(super-)harmonic functions on the whole space, so-called \emph{Liouville property},
has been one of the main concerns of harmonic analysis on various spaces.
One of the most well-known facts about Liouville property is that the non-existence
of non-constant bounded superharmonic functions on the whole space is equivalent to
the recurrence of the corresponding stochastic process. Such an equivalence is
known to hold for standard processes on locally compact separable metrizable spaces
by Blumenthal and Getoor \cite[Chapter II, (4.22)]{BG} and also for more general
right processes by Getoor \cite[Proposition (2.4)]{Get:LMN80}.
Getoor \cite[Proposition 2.14]{Get:Exc} provides the same kind of equivalence
in terms of excessive measures.
The purpose of this paper is to give a completely elementary proof of this
equivalence in the framework of an \emph{arbitrary} symmetric Dirichlet form
on a (non-trivial) $\sigma$-finite measure space. Our proof is purely
functional-analytic and free of topological notions on the state space,
although we need to assume the symmetry of the Dirichlet form.

In the rest of this section, we describe our setting and state the main theorem.
We fix a $\sigma$-finite measure space $(E,\mathcal{B},\meas)$ throughout this paper,
and below all $\mathcal{B}$-measurable functions are assumed to be $[-\infty,\infty]$-valued.
Let $(\mathcal{E},\mathcal{F})$ be a symmetric Dirichlet form on $L^{2}(E,\meas)$ and
let $\{T_{t}\}_{t\in(0,\infty)}$ be its
associated Markovian semigroup on $L^{2}(E,\meas)$. Let
$L_{+}(E,\meas):=\{f\mid f:E\to[0,\infty],\,f\textrm{ is }\mathcal{B}\textrm{-measurable}\}$
and
$L^{0}(E,\meas):=\{f\mid f:E\to\mathbb{R},\,f\textrm{ is }\mathcal{B}\textrm{-measurable}\}$,
where we of course identify any two $\mathcal{B}$-measurable functions which are
equal $\meas$-a.e. Let $\ind{}$ denote the constant function $\ind{}:E\to\{1\}$,
and we regard $\mathbb{R}\ind{}:=\{c\ind{}\mid c\in\mathbb{R}\}$ as
a linear subspace of $L^{0}(E,\meas)$. Also let
$L^{p}_{+}(E,\meas):=L^{p}(E,\meas)\cap L_{+}(E,\meas)$
for $p\in[1,\infty]\cup\{0\}$.
Note that $T_{t}$ is canonically extended to an operator on $L_{+}(E,\meas)$ and
also to a linear operator from
$\mathcal{D}[T_{t}]:=\{f\in L^{0}(E,\meas)\mid T_{t}|f|<\infty\ \meas\textrm{-a.e.}\}$
to $L^{0}(E,\meas)$; see Proposition \ref{prop:T-extension-pos} below.
\begin{dfn}\label{dfn:excessive}
$u\in L_{+}(E,\meas)$ is called \emph{$\mathcal{E}$-excessive} if and only if
$T_{t}u\leq u$ $\meas$-a.e.\ for any $t\in(0,\infty)$. Similarly,
$u\in\bigcap_{t\in(0,\infty)}\mathcal{D}[T_{t}]$ is called
\emph{$\mathcal{E}$-excessive in the wide sense} if and only if
$T_{t}u\leq u$ $\meas$-a.e.\ for any $t\in(0,\infty)$.
\end{dfn}
\begin{rmk}
As stated in \cite{BG,CF,FOT,FT,Shigekawa:convex}, when we call a function $u$
\emph{excessive}, it is usual to assume that \emph{$u$ is non-negative},
which is why we have added \emph{``in the wide sense''}
in the latter part of Definition \ref{dfn:excessive}.
\end{rmk}
$\mathcal{E}$-excessive functions will play the role of superharmonic functions on
the whole state space, and the main theorem of this paper (Theorem \ref{thm:Liouville})
asserts that $(\mathcal{E},\mathcal{F})$ is irreducible and recurrent if and only if
there is no non-constant $\mathcal{E}$-excessive function.

Yet another possible way of formulation of harmonicity of functions (on the whole
space $E$) is to use the extended Dirichlet space $\mathcal{F}_{e}$ associated
with $(\mathcal{E},\mathcal{F})$; $u\in\mathcal{F}_{e}$ could be called
\emph{``superharmonic''} if $\mathcal{E}(u,v)\geq 0$ for any
$v\in\mathcal{F}_{e}\cap L_{+}(E,\meas)$, and
\emph{``harmonic''} if $\mathcal{E}(u,v)=0$ for any $v\in\mathcal{F}_{e}$,
or equivalently, if $\mathcal{E}(u,u)=0$.
In fact, as a key lemma for the proof of the main theorem,
in Proposition \ref{prop:Fe-Tt} below we prove that $u\in\mathcal{F}_{e}$ is ``superharmonic''
in this sense if and only if $u$ is $\mathcal{E}$-excessive in the wide sense.
Under this formulation of harmonicity,
if $(\mathcal{E},\mathcal{F})$ is \emph{recurrent}, i.e.,
$\ind{}\in\mathcal{F}_{e}$ and $\mathcal{E}(\ind{},\ind{})=0$,
then the non-existence of non-constant harmonic
functions amounts to the equality
\begin{equation}\label{eq:irr_ker}
\{u\in\mathcal{F}_{e}\mid\mathcal{E}(u,u)=0\}=\mathbb{R}\ind{}.
\end{equation}

\={O}shima \cite[Theorem 3.1]{Oshima:LMN82} proved \eqref{eq:irr_ker}
(and the completeness of $(\mathcal{F}_{e}/\mathbb{R}\ind{},\mathcal{E})$ as well)
for the Dirichlet form associated with a symmetric Hunt process which is
\emph{recurrent in the sense of Harris}; note that the recurrence in the sense of
Harris is stronger than the usual recurrence of the associated Dirichlet form.
Fukushima and Takeda \cite[Theorem 4.2.4]{FT} (see also \cite[Theorem 2.1.11]{CF})
showed \eqref{eq:irr_ker} for irreducible recurrent symmetric Dirichlet forms
$(\mathcal{E},\mathcal{F})$ under the (only) additional assumption that $\meas(E)<\infty$.
In the recent book \cite{CF}, Chen and Fukushima has extended this result
to the case of $\meas(E)=\infty$ when $(\mathcal{E},\mathcal{F})$ is regular,
by using the theory of random time changes of Dirichlet spaces.
As part of our main theorem, we generalize \eqref{eq:irr_ker} to \emph{any}
irreducible recurrent symmetric Dirichlet form. In fact, this generalization
could be obtained (at least when $L^{2}(E,\meas)$ is separable) also by applying
the theory of regular representations of Dirichlet spaces
(see \cite[Section A.4]{FOT}) to reduce the proof to the case where
$(\mathcal{E},\mathcal{F})$ is regular.
The advantage of our proof is that it is based on totally
elementary analytic arguments and is free from any use of time changes or
regular representations of Dirichlet spaces.

Here is the statement of our main theorem. See \cite[Section 1.1]{CF} or
\cite[Section 1]{F:Fe} for basics on $\mathcal{F}_{e}$, and
\cite[Sections 1.5 and 1.6]{FOT} or \cite[Section 2.1]{CF} for details about
irreducibility and recurrence of $(\mathcal{E},\mathcal{F})$.
We remark that $\mathcal{F}_{e}\subset\bigcap_{t\in(0,\infty)}\mathcal{D}[T_{t}]$
by Lemma \ref{lem:Fe-Tt}-(1) below.
We say that $(E,\mathcal{B},\meas)$ is \emph{non-trivial} if and only if
both $\meas(A)>0$ and $\meas(E\setminus A)>0$ hold for some $A\in\mathcal{B}$,
which is equivalent to the condition that $L^{2}(E,\meas)\not\subset\mathbb{R}\ind{}$
since $(E,\mathcal{B},\meas)$ is assumed to be $\sigma$-finite.
\begin{thm}\label{thm:Liouville}%
Consider the following six conditions.\\
$\textrm{\bfseries\upshape 1)}$
    $(\mathcal{E},\mathcal{F})$ is both irreducible and recurrent.\\
$\textrm{\bfseries\upshape 2)}$
    $\{u\in\mathcal{F}_{e}\mid\mathcal{E}(u,u)=0\}=\mathbb{R}\ind{}$.\\
$\textrm{\bfseries\upshape 3)}$
    $\{u\in\mathcal{F}_{e}\cap L^{\infty}_{+}(E,\meas)\mid\mathcal{E}(u,u)=0\}
    =\{c\ind{}\mid c\in[0,\infty)\}$.\\
$\textrm{\bfseries\upshape 4)}$
    If $u\in\mathcal{F}_{e}$ is $\mathcal{E}$-excessive in the wide sense
    then $u\in\mathbb{R}\ind{}$.\\
$\textrm{\bfseries\upshape 5)}$
    If $u\in L^{0}_{+}(E,\meas)$ is $\mathcal{E}$-excessive
    then $u\in\mathbb{R}\ind{}$.\\
$\textrm{\bfseries\upshape 6)}$
    If $u\in\mathcal{F}_{e}\cap L^{\infty}_{+}(E,\meas)$ is
    $\mathcal{E}$-excessive then $u\in\mathbb{R}\ind{}$.\\
The three conditions
$\textrm{\bfseries\upshape 1)},\textrm{\bfseries\upshape 2)},\textrm{\bfseries\upshape 3)}$
are equivalent to each other and imply
$\textrm{\bfseries\upshape 4)},\textrm{\bfseries\upshape 5)},\textrm{\bfseries\upshape 6)}$.
 If $(E,\mathcal{B},\meas)$ is non-trivial, then the six conditions are all equivalent.
\end{thm}

The organization of this paper is as follows. In Section \ref{sec:preliminaries},
we prepare basic results about the extended space $\mathcal{F}_{e}$ and
$\mathcal{E}$-excessive functions, which are valid as long as
$(\mathcal{E},\mathcal{F})$ is a symmetric positivity preserving form.
The key results there are Propositions \ref{prop:Fe-Tt} and \ref{prop:excessive-char},
which are essentially known but seem new in the present general framework.
Furthermore Proposition \ref{prop:excessive-char} provides a characterization of
the notion of $\mathcal{E}$-excessive functions in terms of $\mathcal{F}_{e}$ and $\mathcal{E}$.
Making use of these two propositions, we show Theorem \ref{thm:Liouville}
in Section \ref{sec:proof}.
\section{Preliminaries: the extended (Dirichlet) space and excessive functions}\label{sec:preliminaries}
As noted in the previous section, we fix a $\sigma$-finite measure space
$(E,\mathcal{B},\meas)$ throughout this paper, and all
$\mathcal{B}$-measurable functions are assumed to be $[-\infty,\infty]$-valued.
Note that by the $\sigma$-finiteness of $(E,\mathcal{B},\meas)$ we can take
$\eta\in L^{1}(E,\meas)\cap L^{\infty}(E,\meas)$ such that $\eta>0$ $\meas$-a.e.
\begin{ntn}
\textup{(0)} We follow the convention that $\mathbb{N}=\{1,2,3,\dots\}$,
i.e., $0\not\in\mathbb{N}$.

\noindent
\textup{(1)} For $a,b\in[-\infty,\infty]$, we write $a\vee b:=\max\{a,b\}$,
$a\wedge b:=\min\{a,b\}$, $a^{+}:=a\vee 0$ and
$a^{-}:=-(a\wedge 0)$. For $\{a_{n}\}_{n\in\mathbb{N}}\subset[-\infty,\infty]$
and $a\in[-\infty,\infty]$,
we write $a_{n}\uparrow a$ (resp.\ $a_{n}\downarrow a$) if and only if
$\{a_{n}\}_{n\in\mathbb{N}}$ is non-decreasing (resp.\ non-increasing) and
$\lim_{n\to\infty}a_{n}=a$. We use the same notation
also for ($\meas$-equivalence classes of) $[-\infty,\infty]$-valued functions.

\noindent
\textup{(2)} As introduced before Definition \ref{dfn:excessive},
identifying any two $\mathcal{B}$-measurable functions that are equal $\meas$-a.e., we set
$L_{+}(E,\meas):=\{f\mid f:E\to[0,\infty],\,f\textrm{ is }\mathcal{B}\textrm{-measurable}\}$,
$L^{0}(E,\meas):=\{f\mid f:E\to\mathbb{R},\,f\textrm{ is }\mathcal{B}\textrm{-measurable}\}$
and $L^{p}_{+}(E,\meas):=L^{p}(E,\meas)\cap L_{+}(E,\meas)$, $p\in[1,\infty]\cup\{0\}$.
We regard $\mathbb{R}\ind{}:=\{c\ind{}\mid c\in\mathbb{R}\}$ as
a linear subspace of $L^{0}(E,\meas)$.
Let $\|\cdot\|_{p}$ denote the norm of $L^{p}(E,\meas)$ for $p\in[1,\infty]$.
Finally, let $\langle f,g\rangle:=\int_{E}fg\,d\meas$ for
$f,g\in L_{+}(E,\meas)$ and also for
$f,g\in L^{0}(E,\meas)$ with $fg\in L^{1}(E,\meas)$.
\end{ntn}

Recall the following definitions regarding bounded linear operators on $L^{2}(E,\meas)$.
\begin{dfn}\label{dfn:pp-Markov}
Let $T:L^{2}(E,\meas)\to L^{2}(E,\meas)$ be a bounded linear operator on
$L^{2}(E,\meas)$.

\noindent
\textup{(1)} $T$ is called \emph{positivity preserving} if and only if
$Tf\geq 0$ $\meas$-a.e.\ for any $f\in L^{2}_{+}(E,\meas)$.

\noindent
\textup{(2)} $T$ is called \emph{Markovian} if and only if
$0\leq Tf\leq 1$ $\meas$-a.e.\ for any $f\in L^{2}(E,\meas)$ with
$0\leq f\leq 1$ $\meas$-a.e.
\end{dfn}
Clearly, if $T$ is positivity preserving then so is its adjoint $T^{*}$. Note that
if $T$ is Markovian, then it is positivity preserving, $\|Tf\|_{\infty}\leq\|f\|_{\infty}$
for any $L^{2}(E,\meas)\cap L^{\infty}(E,\meas)$ and
$\|T^{*}f\|_{1}\leq\|f\|_{1}$ for any $f\in L^{1}(E,\meas)\cap L^{2}(E,\meas)$.
Moreover, using the $\sigma$-finiteness of $(E,\mathcal{B},\meas)$,
we easily have the following proposition.
\begin{prop}\label{prop:T-extension-pos}
Let $T:L^{2}(E,\meas)\to L^{2}(E,\meas)$ be a positivity preserving bounded
linear operator on $L^{2}(E,\meas)$.

\noindent
\textup{(1)} $T|_{L^{2}_{+}(E,\meas)}$ uniquely extends to a map
$T:L_{+}(E,\meas)\to L_{+}(E,\meas)$ such that
$Tf_{n}\uparrow Tf$ $\meas$-a.e.\ for any $f\in L_{+}(E,\meas)$ and
any $\{f_{n}\}_{n\in\mathbb{N}}\subset L_{+}(E,\meas)$ with
$f_{n}\uparrow f$ $\meas$-a.e.
Moreover, let $f,g\in L_{+}(E,\meas)$ and $a\in[0,\infty]$. Then
$T(f+g)=Tf+Tg$, $T(af)=aTf$,
$\langle Tf,g\rangle=\langle f,T^{*}g\rangle$,
and if $f\leq g$ $\meas$-a.e.\ then $Tf\leq Tg$ $\meas$-a.e.

\noindent
\textup{(2)} Let $\mathcal{D}[T]:=\{f\in L^{0}(E,\meas)\mid T|f|<\infty\ \meas\textrm{-a.e.}\}$.
Then $T:L^{2}(E,\meas)\to L^{2}(E,\meas)$ is extended to a linear operator
$T:\mathcal{D}[T]\to L^{0}(E,\meas)$ given by
$Tf:=T(f^{+})-T(f^{-})$, $f\in\mathcal{D}[T]$, so that it has the following properties:\\
\textup{(i)} If $f,g\in\mathcal{D}[T]$ and $f\leq g$ $\meas$-a.e.\ then
$Tf\leq Tg$ $\meas$-a.e.\\
\textup{(ii)} If $\{f_{n}\}_{n\in\mathbb{N}}\subset\mathcal{D}[T]$ and
$f,g\in\mathcal{D}[T]$ satisfy $\lim_{n\to\infty}f_{n}=f$ $\meas$-a.e.\ and
$|f_{n}|\leq |g|$ $\meas$-a.e.\ for any $n\in\mathbb{N}$, then
$\lim_{n\to\infty}Tf_{n}=Tf$ $\meas$-a.e.
\end{prop}
%
%\begin{proof}
%(1) By the $\sigma$-finiteness we can choose
%$\{\eta_{n}\}_{n\in\mathbb{N}}\subset L^{1}(E,\meas)$ so that
%$0\leq\eta_{n}\leq 1$ $\meas$-a.e., $n\in\mathbb{N}$ and $\eta_{n}\uparrow 1$
%$\meas$-a.e. Define $Tf$ for $f\in L_{+}(E,\meas)$ so that
%$T(\eta_{n}(f\wedge n))\uparrow Tf$ $\meas$-a.e.
%Then the standard arguments show that this operator
%$T:L_{+}(E,\meas)\to L_{+}(E,\meas)$ have the desired properties.
%The uniqueness of such $T$ is clear.
%
%\noindent
%(2) This is easily shown by the standard arguments based on (1).
%\end{proof}
%
%\begin{prop}\label{prop:T-extension}
%Let $T$ be a Markovian symmetric bounded linear operator on $L^{2}(E,\meas)$.
%Then $L^{1}\vee L^{\infty}(E,\meas)\subset\mathcal{D}[T]$ and
%$\|Tf\|_{p}\leq \|f\|_{p}$ for any $f\in L^{p}(E,\meas)$, $p=1,\infty$. Moreover
%for $f\in L^{1}\vee L^{\infty}(E,\meas)$, $Tf$ is the unique element of $L^{1}\vee L^{\infty}(E,\meas)$
%that satisfies $\langle Tf,g\rangle=\langle f,Tg\rangle$
%for any $g\in L^{1}(E,\meas)\cap L^{\infty}(E,\meas)$.
%In particular, $T$ defines a linear operator
%$T:L^{1}\vee L^{\infty}(E,\meas)\to L^{1}\vee L^{\infty}(E,\meas)$.
%\end{prop}
%
%\begin{proof}
%This is immediate from $\|Tf\|_{p}\leq\|f\|_{p}$ for
%$f\in L^{2}(E,\meas)\cap L^{p}(E,\meas)$, $p=1,\infty$.
%\end{proof}
%
Throughout the rest of this paper, we fix a closed symmetric form
$(\mathcal{E},\mathcal{F})$ on $L^{2}(E,\meas)$ together with
its associated symmetric strongly continuous contraction semigroup
$\{T_{t}\}_{t\in(0,\infty)}$ and resolvent $\{G_{\alpha}\}_{\alpha\in(0,\infty)}$
on $L^{2}(E,\meas)$; see \cite[Chapter 1.3]{FOT} for basics on closed
symmetric forms on Hilbert spaces and their associated semigroups and resolvents.
%(\emph{i.e.}, $\mathcal{E}$ is a non-negative definite symmetric bilinear form
%$\mathcal{E}:\mathcal{F}\times\mathcal{F}\to\mathbb{R}$
%defined on a dense linear subspace $\mathcal{F}$ of $L^{2}(E,\meas)$ such that
%$(\mathcal{F},\mathcal{E}+\langle\cdot,\cdot\rangle)$ is a Hilbert space).

Let us further recall the following definition.
\begin{dfn}\label{dfn:pform}
\textup{(1)} $(\mathcal{E},\mathcal{F})$ is called a \emph{positivity preserving form}
if and only if $u^{+}\in\mathcal{F}$ and $\mathcal{E}(u^{+},u^{+})\leq\mathcal{E}(u,u)$
for any $u\in\mathcal{F}$, or equivalently,
$T_{t}$ is positivity preserving for any $t\in(0,\infty)$.

\noindent
\textup{(2)} $(\mathcal{E},\mathcal{F})$ is called a \emph{Dirichlet form} if and only if
$u^{+}\wedge 1\in\mathcal{F}$ and $\mathcal{E}(u^{+}\wedge 1,u^{+}\wedge 1)\leq\mathcal{E}(u,u)$
for any $u\in\mathcal{F}$, or equivalently,
$T_{t}$ is Markovian for any $t\in(0,\infty)$.
\end{dfn}
See, e.g., \cite[Section 2]{Ouh:PA96} for the equivalences stated in Definition \ref{dfn:pform}.

In the rest of this section, we assume that $(\mathcal{E},\mathcal{F})$ is a
positivity preserving form. The following definition is standard
(see \cite[Definition 3]{Sch:extend}, \cite[Definition 1.1.4]{CF} or \cite[Definition 1.4]{F:Fe}).
\begin{dfn}\label{dfn:Fe}
We define the \emph{extended space $\mathcal{F}_{e}$ associated with
$(\mathcal{E},\mathcal{F})$} by
\begin{equation}\label{eq:Fe}
\mathcal{F}_{e}:=\biggl\{u\in L^{0}(E,\meas)\ \biggm|
  \begin{minipage}{230pt}
    $\lim_{n\to\infty}u_{n}=u$ $\meas$-a.e.\ for some
    $\{u_{n}\}_{n\in\mathbb{N}}\subset\mathcal{F}$ with
    $\lim_{k\wedge\ell\to\infty}\mathcal{E}(u_{k}-u_{\ell},u_{k}-u_{\ell})=0$
  \end{minipage}\biggr\}.
\end{equation}
For $u\in\mathcal{F}_{e}$, such
$\{u_{n}\}_{n\in\mathbb{N}}\subset\mathcal{F}$ as in \eqref{eq:Fe} is
called an \emph{approximating sequence for $u$}.
When $(\mathcal{E},\mathcal{F})$ is a Dirichlet form, $\mathcal{F}_{e}$ is called
the \emph{extended Dirichlet space associated with $(\mathcal{E},\mathcal{F})$}.
\end{dfn}
Obviously $\mathcal{F}\subset\mathcal{F}_{e}$ and $\mathcal{F}_{e}$ is a linear
subspace of $L^{0}(E,\meas)$. By virtue of \cite[Proposition 2]{Sch:Fatou},
$\mathcal{F}=\mathcal{F}_{e}\cap L^{2}(E,\meas)$, and
for $u,v\in\mathcal{F}_{e}$ with approximating sequences
$\{u_{n}\}_{n\in\mathbb{N}}$ and $\{v_{n}\}_{n\in\mathbb{N}}$, respectively,
the limit $\lim_{n\to\infty}\mathcal{E}(u_{n},v_{n})\in\mathbb{R}$
exists and is independent of particular choices of
$\{u_{n}\}_{n\in\mathbb{N}}$ and $\{v_{n}\}_{n\in\mathbb{N}}$,
as discussed in \cite[before Definition 3]{Sch:extend}.
By setting $\mathcal{E}(u,v):=\lim_{n\to\infty}\mathcal{E}(u_{n},v_{n})$,
$\mathcal{E}$ is extended to a non-negative definite symmetric bilinear form on
$\mathcal{F}_{e}$. Then it is easy to see that
$\lim_{n\to\infty}\mathcal{E}(u-u_{n},u-u_{n})=0$ for
$u\in\mathcal{F}_{e}$ and any approximating sequence
$\{u_{n}\}_{n\in\mathbb{N}}\subset\mathcal{F}$ for $u$.
Moreover, we have the following proposition due to Schmuland \cite{Sch:extend}, which is
easily proved by utilizing a version \cite[Theorem A.4.1-(ii)]{CF} of the Banach-Saks theorem.
\begin{prop}[{\cite[Lemma 2]{Sch:extend}}]\label{prop:Fe-Fatou}
\hspace*{-1.85pt}Let $u\in L^{0}(E,\meas)$ and $\{u_{n}\}_{n\in\mathbb{N}}\subset\mathcal{F}$
satisfy $\lim_{n\to\infty}u_{n}$\\$=u$ $\meas$-a.e.\ and
$\liminf_{n\to\infty}\mathcal{E}(u_{n},u_{n})<\infty$.
Then $u\in\mathcal{F}_{e}$,
$\mathcal{E}(u,u)\leq\liminf_{n\to\infty}\mathcal{E}(u_{n},u_{n})$, and
$\liminf_{n\to\infty}\mathcal{E}(u_{n},v)\leq\mathcal{E}(u,v)
	\leq\limsup_{n\to\infty}\mathcal{E}(u_{n},v)$
for any $v\in\mathcal{F}_{e}$.
\end{prop}
In particular, we easily see from Proposition \ref{prop:Fe-Fatou} that
$u^{+}\in\mathcal{F}_{e}$ and $\mathcal{E}(u^{+},u^{+})\leq\mathcal{E}(u,u)$
for any $u\in\mathcal{F}_{e}$.
\begin{rmk}
For symmetric Dirichlet forms, the properties of $\mathcal{F}_{e}$ stated above are well-known
and most of them are proved in the textbooks \cite[Section 1.1]{CF} and \cite[Section 4.1]{FT}
and also in \cite[Section 1]{F:Fe}. In fact, we can verify similar results in a quite general
setting; see Schmuland \cite{Sch:extend} for details.
\end{rmk}
%
%\begin{ntn}
%We write $\|u\|_{\mathcal{E}}:=\mathcal{E}(u,u)^{1/2}$ for $u\in\mathcal{F}_{e}$.
%\end{ntn}
%
The next proposition (Proposition \ref{prop:Fe-Tt} below) requires the following lemmas.
\begin{lem}\label{lem:Fe-norm}
Let $\eta\in L^{1}(E,\meas)\cap L^{2}(E,\meas)$ be such that $\eta>0$ $\meas$-a.e.,
and set
$\|u\|_{\mathcal{F}_{e}}:=\mathcal{E}(u,u)^{1/2}+\int_{E}(|u|\wedge 1)\eta\,d\meas$
for $u\in\mathcal{F}_{e}$. Then we have the following assertions:

\noindent
\textup{(1)} $\|u+v\|_{\mathcal{F}_{e}}\leq\|u\|_{\mathcal{F}_{e}}+\|v\|_{\mathcal{F}_{e}}$
and $\|au\|_{\mathcal{F}_{e}}\leq(|a|\vee 1)\|u\|_{\mathcal{F}_{e}}$
for any $u,v\in\mathcal{F}_{e}$ and any $a\in\mathbb{R}$.

\noindent
\textup{(2)} $\mathcal{F}_{e}$ is a complete metric space under the metric
$d_{\mathcal{F}_{e}}$ given by
$d_{\mathcal{F}_{e}}(u,v):=\|u-v\|_{\mathcal{F}_{e}}$.
\end{lem}
\begin{proof}
(1) is immediate and $d_{\mathcal{F}_{e}}$ is clearly a metric on $\mathcal{F}_{e}$.
For the proof of its completeness,
let $\{u_{n}\}_{n\in\mathbb{N}}\subset\mathcal{F}_{e}$
be a Cauchy sequence in $(\mathcal{F}_{e},d_{\mathcal{F}_{e}})$.
Noting that $\mathcal{F}$ is dense in $(\mathcal{F}_{e},d_{\mathcal{F}_{e}})$,
for each $n\in\mathbb{N}$ take $v_{n}\in\mathcal{F}$ such that
$\|v_{n}-u_{n}\|_{\mathcal{F}_{e}}\leq n^{-1}$. Then
$\{v_{n}\}_{n\in\mathbb{N}}$ is also a Cauchy sequence in
$(\mathcal{F}_{e},d_{\mathcal{F}_{e}})$.
A Borel-Cantelli argument easily yields a subsequence
$\{v_{n_{k}}\}_{k\in\mathbb{N}}$ of $\{v_{n}\}_{n\in\mathbb{N}}$
converging $\meas$-a.e.\ to some $u\in L^{0}(E,\meas)$, which means that
$u\in\mathcal{F}_{e}$ with approximating sequence
$\{v_{n_{k}}\}_{k\in\mathbb{N}}$ and hence that
$\lim_{k\to\infty}\|u-v_{n_{k}}\|_{\mathcal{F}_{e}}=0$.
The same argument also implies that every subsequence of $\{v_{n}\}_{n\in\mathbb{N}}$
admits a further subsequence converging to $u$ in
$(\mathcal{F}_{e},d_{\mathcal{F}_{e}})$,
from which $\lim_{n\to\infty}\|u-v_{n}\|_{\mathcal{F}_{e}}=0$ follows.
Thus $\lim_{n\to\infty}\|u-u_{n}\|_{\mathcal{F}_{e}}=0$.
\end{proof}
\begin{lem}\label{lem:Fe-Tt}
\textup{(1)} $\mathcal{F}_{e}\subset\bigcap_{t\in(0,\infty)}\mathcal{D}[T_{t}]$ and
$T_{t}(\mathcal{F}_{e})\subset\mathcal{F}_{e}$ for any $t\in(0,\infty)$.

\noindent
\textup{(2)} Let $\eta$ and $\|\cdot\|_{\mathcal{F}_{e}}$ be as in Lemma \textup{\ref{lem:Fe-norm}},
and let $u\in\mathcal{F}_{e}$. Then
$\mathcal{E}(T_{t}u,T_{t}u)\leq\mathcal{E}(u,u)$,
$\|u-T_{t}u\|_{2}^{2}\leq t\mathcal{E}(u,u)$ and
$\|T_{t}u\|_{\mathcal{F}_{e}}\leq(3+\|\eta\|_{2}\sqrt{t})\|u\|_{\mathcal{F}_{e}}$
for any $t\in(0,\infty)$, $T_{s}T_{t}u=T_{s+t}u$ for any $s,t\in(0,\infty)$,
and $\lim_{t\downarrow 0}\|u-T_{t}u\|_{\mathcal{F}_{e}}=0$.
\end{lem}
\begin{proof}
Let $\eta$, $\|\cdot\|_{\mathcal{F}_{e}}$ and $d_{\mathcal{F}_{e}}$
be as in Lemma \ref{lem:Fe-norm}.
First we prove (2) for $u\in\mathcal{F}$. The fourth assertion is clear.
$T_{t}u\in\mathcal{F}$ and $\mathcal{E}(T_{t}u,T_{t}u)\leq\mathcal{E}(u,u)$
for $t\in(0,\infty)$ by \cite[Lemma 1.3.3-(i)]{FOT}, and
$\lim_{t\downarrow 0}\|u-T_{t}u\|_{\mathcal{F}_{e}}=0$ by \cite[Lemma 1.3.3-(iii)]{FOT}.
Let $t\in(0,\infty)$. Noting that
$\langle f-T_{t}f,T_{t}f\rangle=\|T_{t/2}f\|_{2}^{2}-\|T_{t}f\|_{2}^{2}\geq 0$
for $f\in L^{2}(E,\meas)$, we have
$\|u-T_{t}u\|_{2}^{2}=\langle u-T_{t}u,u\rangle-\langle u-T_{t}u,T_{t}u\rangle
	\leq\langle u-T_{t}u,u\rangle\leq t\mathcal{E}(u,u)$
by \cite[Lemma 1.3.4-(i)]{FOT}. Applying these estimates to
$\|u-T_{t}u\|_{\mathcal{F}_{e}}
	\leq\mathcal{E}(u,u)^{1/2}+\mathcal{E}(T_{t}u,T_{t}u)^{1/2}+\|\eta\|_{2}\|u-T_{t}u\|_{2}$
easily yields
$\|T_{t}u\|_{\mathcal{F}_{e}}\leq(3+\|\eta\|_{2}\sqrt{t})\|u\|_{\mathcal{F}_{e}}$.

Now since $\mathcal{F}$ is dense in a complete metric space
$(\mathcal{F}_{e},d_{\mathcal{F}_{e}})$, it follows from the previous paragraph
that $T_{t}|_{\mathcal{F}}$ is uniquely extended to a continuous map
$T^{e}_{t}$ from $(\mathcal{F}_{e},d_{\mathcal{F}_{e}})$ to itself, and
then clearly $T^{e}_{t}$ is linear and the assertions of (2) are true
with $T^{e}_{t}$ in place of $T_{t}$.

Let $t\in(0,\infty)$ and $u\in\mathcal{F}_{e}\cap L_{+}(E,\meas)$.
It remains to show $T^{e}_{t}u=T_{t}u$, as $v^{+},v^{-}\in\mathcal{F}_{e}$
for $v\in\mathcal{F}_{e}$. Since
$v^{+}\wedge u\in\mathcal{F}_{e}\cap L^{2}(E,\meas)=\mathcal{F}$ and
$\mathcal{E}(v^{+}\wedge u,v^{+}\wedge u)^{1/2}\leq\mathcal{E}(v,v)^{1/2}+\mathcal{E}(u,u)^{1/2}$
for any $v\in\mathcal{F}$ by the positivity preserving property of $(\mathcal{E},\mathcal{F})$,
an application of the Banach-Saks theorem \cite[Theorem A.4.1-(ii)]{CF} assures
the existence of an approximating sequence $\{w_{n}\}_{n\in\mathbb{N}}$ for $u$
such that $0\leq w_{n}\leq u$ $\meas$-a.e. A Borel-Cantelli argument yields
a subsequence $\{w_{n_{k}}\}_{k\in\mathbb{N}}$ such that
$\lim_{k\to\infty}T_{t}w_{n_{k}}=T^{e}_{t}u$ $\meas$-a.e.,
and $T^{e}_{t}u=T_{t}u$ follows by letting $k\to\infty$ in
$T_{t}(\inf_{j\geq k}w_{n_{j}})\leq T_{t}w_{n_{k}}\leq T_{t}u$ $\meas$-a.e.
\end{proof}

The following proposition (Proposition \ref{prop:Fe-Tt}), which seems new
in spite of its easiness, plays an essential role in the proof of
$\textrm{\bfseries\upshape 1)}\Rightarrow\textrm{\bfseries\upshape 2)}$
of Theorem \ref{thm:Liouville}. Proposition \ref{prop:Fe-Tt}-(2) is an extension
of a result of Chen and Kuwae \cite[Lemma 3.1]{CK:subh} for functions in $\mathcal{F}$
to those in $\mathcal{F}_{e}$, and Proposition \ref{prop:Fe-Tt}-(3) extends
a basic fact for functions in $\mathcal{F}$ to those in $\mathcal{F}_{e}$.
\begin{prop}\label{prop:Fe-Tt}
\textup{(1)} Let $u\in\mathcal{F}_{e}$ and $v\in\mathcal{F}$. Then
\begin{equation}\label{eq:Fe-Tt}
\lim_{t\downarrow 0}\frac{1}{t}\langle u-T_{t}u,v\rangle
    =\mathcal{E}(u,v)
    \mspace{24mu}\textrm{and}\mspace{24mu}
    \langle u-T_{t}u,v\rangle=\int_{0}^{t}\mathcal{E}(u,T_{s}v)ds,
    \mspace{12mu}t\in(0,\infty).
\end{equation}

\noindent
\textup{(2)} Let $u\in\mathcal{F}_{e}$. Then $u$ is $\mathcal{E}$-excessive in the wide sense
if and only if $\mathcal{E}(u,v)\geq 0$ for any $v\in\mathcal{F}\cap L_{+}(E,\meas)$,
or equivalently, for any $v\in\mathcal{F}_{e}\cap L_{+}(E,\meas)$.

\noindent
\textup{(3)} Let $u\in\mathcal{F}_{e}$. Then $T_{t}u=u$ for any $t\in(0,\infty)$
if and only if $\mathcal{E}(u,u)=0$.
\end{prop}
\begin{proof}
(1) Let $u\in\mathcal{F}_{e}$, $v\in\mathcal{F}$ and
set $\varphi(t):=\langle u-T_{t}u,v\rangle$ for $t\in[0,\infty)$, where $T_{0}u:=u$.
Then $t^{-1}|\varphi(t)|\leq\mathcal{E}(u,u)^{1/2}\mathcal{E}(v,v)^{1/2}$ for
$t\in(0,\infty)$ and $\lim_{t\downarrow 0}t^{-1}\varphi(t)=\mathcal{E}(u,v)$
if $u\in\mathcal{F}$ by \cite[Lemma 1.3.4-(i)]{FOT}, and the same are true
for $u\in\mathcal{F}_{e}$ as well by Lemma \ref{lem:Fe-Tt}.
Using Lemma \ref{lem:Fe-Tt}, we easily see also that
$\varphi'(t)=\mathcal{E}(u,T_{t}v)$ for $t\in[0,\infty)$ and that
$\varphi'$ is continuous on $[0,\infty)$, proving \eqref{eq:Fe-Tt}.

\noindent
(2) The third assertion of Proposition \ref{prop:Fe-Fatou} together with the positivity
preserving property of $(\mathcal{E},\mathcal{F})$ easily implies that
$\mathcal{E}(u,v)\geq 0$ for any $v\in\mathcal{F}\cap L_{+}(E,\meas)$ if and only if
the same is true for any $v\in\mathcal{F}_{e}\cap L_{+}(E,\meas)$.
The rest of the assertion is immediate from \eqref{eq:Fe-Tt}.

\noindent
(3) This is an immediate consequence of (2).
\end{proof}
The next proposition (Proposition \ref{prop:excessive-char}),
which characterizes the notion of $\mathcal{E}$-excessive functions
in terms of $\mathcal{F}_{e}$ and $\mathcal{E}$, is of independent interest.
The proof is based on a result \cite[Corollary 2.4]{Ouh:PA96} of Ouhabaz which
provides a characterization of invariance of closed convex sets for semigroups
on Hilbert spaces. A similar argument in a more general framework can be
found in Shigekawa \cite{Shigekawa:convex}.
\begin{prop}\label{prop:excessive-char}
Let $u\in L_{+}(E,\meas)$. Then $u$ is $\mathcal{E}$-excessive if and only if
$v\wedge u\in\mathcal{F}_{e}$ and $\mathcal{E}(v\wedge u,v\wedge u)\leq\mathcal{E}(v,v)$
for any $v\in\mathcal{F}_{e}$.
\end{prop}
\begin{crl}\label{cor:excessive-char}
The notion of $\mathcal{E}$-excessive functions 
is determined solely by the pair $(\mathcal{F}_{e},\mathcal{E})$ of the extended space
$\mathcal{F}_{e}$ and the form
$\mathcal{E}:\mathcal{F}_{e}\times\mathcal{F}_{e}\to\mathbb{R}$.
\end{crl}
\begin{crl}\label{cor:excessive}
Let $u\in L_{+}(E,\meas)$ be $\mathcal{E}$-excessive and $v\in\mathcal{F}_{e}$.
Suppose $u\leq v$ $\meas$-a.e. Then $u\in\mathcal{F}_{e}$ and
$\mathcal{E}(u,u)\leq\mathcal{E}(v,v)$.
\end{crl}
\begin{rmk}\label{rmk:excessive}
Chen and Kuwae \cite[Lemma 3.3]{CK:subh} gave a probabilistic proof of
Corollary \textup{\ref{cor:excessive}} for the Dirichlet forms
associated with symmetric right Markov processes.
\end{rmk}
\begin{proof}[Proof of Proposition \textup{\ref{prop:excessive-char}}]
Let $K_{u}:=\{f\in L^{2}(E,\meas)\mid f\leq u\ \meas\textrm{-a.e.}\}$, which is clearly
a closed convex subset of $L^{2}(E,\meas)$. We claim that
\begin{equation}\label{eq:excessive-K}
\textrm{$u$ is $\mathcal{E}$-excessive}\quad\textrm{if and only if}\quad
  \textrm{$T_{t}(K_{u})\subset K_{u}$ for any $t\in(0,\infty)$.}
\end{equation}
Indeed, let $t\in(0,\infty)$.
If $T_{t}u\leq u$ $\meas$-a.e.\ then $T_{t}f\leq T_{t}u\leq u$ $\meas$-a.e.\ for any $f\in K_{u}$
and hence $T_{t}(K_{u})\subset K_{u}$.
Conversely if $T_{t}(K_{u})\subset K_{u}$, then choosing $\eta\in L^{2}(E,\meas)$ so that
$\eta>0$ $\meas$-a.e., we have $(n\eta)\wedge u\uparrow u$ $\meas$-a.e.,
$(n\eta)\wedge u\in K_{u}$ and hence $T_{t}u=\lim_{n\to\infty}T_{t}((n\eta)\wedge u)\leq u$
$\meas$-a.e.

On the other hand, since the projection of $f\in L^{2}(E,\meas)$ on $K_{u}$ is
given by $f\wedge u$, \cite[Corollary 2.4]{Ouh:PA96} tells us that
$T_{t}(K_{u})\subset K_{u}$ for any $t\in(0,\infty)$ if and only if
\begin{equation}\label{eq:excessive-F}
\textrm{$v\wedge u\in\mathcal{F}$}\quad\textrm{and}\quad
  \textrm{$\mathcal{E}(v\wedge u,v\wedge u)\leq\mathcal{E}(v,v)$}
  \qquad\textrm{for any $v\in\mathcal{F}$.}
\end{equation}
Finally, $\mathcal{F}_{e}\cap L^{2}(E,\meas)=\mathcal{F}$ and
Proposition \ref{prop:Fe-Fatou} easily imply that
\eqref{eq:excessive-F} is equivalent to the same condition
with $\mathcal{F}_{e}$ in place of $\mathcal{F}$, completing the proof.
\end{proof}
\section{Proof of Theorem \ref{thm:Liouville}}\label{sec:proof}
We are now ready for the proof of Theorem \ref{thm:Liouville}. We assume throughout
this section that our closed symmetric form $(\mathcal{E},\mathcal{F})$ is a Dirichlet form.
The proof consists of three steps. The first one is
Proposition \ref{prop:prf-step1} below, which establishes
$\textrm{\bfseries\upshape 1)}\Rightarrow\textrm{\bfseries\upshape 2)}$ of
Theorem \ref{thm:Liouville} and whose proof makes full use of
Proposition \ref{prop:Fe-Tt}-(3).
Recall the following notions concerning the irreducibility of $(\mathcal{E},\mathcal{F})$;
see \cite[Section 1.6]{FOT} or \cite[Section 2.1]{CF} for details.
\begin{dfn}\label{dfn:irreducible}
\textup{(1)} A set $A\in\mathcal{B}$ is called \emph{$\mathcal{E}$-invariant} if and only if
$\ind{A}T_{t}(f\ind{E\setminus A})=0$ $\meas$-a.e.\ for any $f\in L^{2}(E,\meas)$ and
any $t\in(0,\infty)$.

\noindent
\textup{(2)} $(\mathcal{E},\mathcal{F})$ is called \emph{irreducible} if and only if
either $\meas(A)=0$ or $\meas(E\setminus A)=0$ holds
for any $\mathcal{E}$-invariant $A\in\mathcal{B}$.
\end{dfn}
\begin{lem}\label{lem:excessive-ppforms}
Let $u\in L_{+}(E,\meas)$ be $\mathcal{E}$-excessive. Then
$\{u=0\}$ is $\mathcal{E}$-invariant.
\end{lem}
\begin{proof}
In fact, the following proof is valid as long as $(\mathcal{E},\mathcal{F})$
is a symmetric positivity preserving form.
Let $B:=\{u=0\}$, $f\in L^{2}(E,\meas)$ and set
$f_{n}:=|f|\wedge(nu)$ for $n\in\mathbb{N}$, so that
$f_{n}\uparrow|f|\ind{E\setminus B}$ $\meas$-a.e. Then
$0\leq\ind{B}T_{t}f_{n}\leq\ind{B}T_{t}(nu)\leq n\ind{B}u=0$
$\meas$-a.e., and letting $n\to\infty$ leads to
$|\ind{B}T_{t}(f\ind{E\setminus B})|\leq\ind{B}T_{t}(|f|\ind{E\setminus B})=0$
$\meas$-a.e. Thus $B=\{u=0\}$ is $\mathcal{E}$-invariant.
\end{proof}
\begin{prop}\label{prop:prf-step1}
Suppose that $(\mathcal{E},\mathcal{F})$ is irreducible. If $u\in\mathcal{F}_{e}$
and $\mathcal{E}(u,u)=0$ then $u\in\mathbb{R}\ind{}$.
\end{prop}
\begin{proof}
We follow \cite[Proof of Theorem 2.1.11,\,(i)\,$\Rightarrow$\,(ii)]{CF}.
Let $u\in\mathcal{F}_{e}$ satisfy $\mathcal{E}(u,u)=0$.
We may assume that $\meas(\{u>0\})>0$. Let $\lambda\in[0,\infty)$ and
$u_{\lambda}:=u-u\wedge\lambda$. Since $(\mathcal{E},\mathcal{F})$ is assumed
to be a Dirichlet form, $u_{\lambda}\in\mathcal{F}_{e}\cap L_{+}(E,\meas)$ and
$\mathcal{E}(u_{\lambda},u_{\lambda})=0$ (see Proposition \ref{prop:excessive-char}),
and therefore $T_{t}u_{\lambda}=u_{\lambda}$ for any $t\in(0,\infty)$
by Proposition \ref{prop:Fe-Tt}-(3). Then $\{u_{\lambda}=0\}$
is $\mathcal{E}$-invariant by Lemma \ref{lem:excessive-ppforms},
and the irreducibility of $(\mathcal{E},\mathcal{F})$ implies that
either $\meas(\{u_{\lambda}=0\})=0$ or $\meas(\{u_{\lambda}>0\})=0$ holds.
Now setting $\kappa:=\sup\{\lambda\in[0,\infty)\mid \meas(\{u_{\lambda}=0\})=0\}$,
we easily see that $\kappa\in(0,\infty)$ and that $u=\kappa$ $\meas$-a.e.
\end{proof}
For the rest of the proof of Theorem \ref{thm:Liouville}, let us recall basic notions
concerning recurrence and transience of Dirichlet forms.
See \cite[Sections 1.5 and 1.6]{FOT} or \cite[Section 2.1]{CF} for details.
For $t\in(0,\infty)$, we define $S_{t}:L^{2}(E,\meas)\to L^{2}(E,\meas)$ by
$S_{t}f:=\int_{0}^{t}T_{s}f\,ds$, where the integral is the Riemann integral in
$L^{2}(E,\meas)$. Then $t^{-1}S_{t}$ is a Markovian symmetric bounded linear
operator on $L^{2}(E,\meas)$, and therefore it is canonically extended to
an operator on $L_{+}(E,\meas)$ by Proposition \ref{prop:T-extension-pos}.
Furthermore, for any $s,t\in(0,\infty)$ we easily see that
$S_{s+t}=S_{s}+T_{s}S_{t}=S_{s}+S_{t}T_{s}$ as operators
on $L_{+}(E,\meas)$ or on $L^{2}(E,\meas)$.

Let $f\in L_{+}(E,\meas)$. Then $0\leq S_{s}f\leq S_{t}f$ $\meas$-a.e.\ and
$0\leq G_{\beta}f\leq G_{\alpha}f$ $\meas$-a.e.\ for $0<s<t$, $0<\alpha<\beta$.
Therefore there exists a unique $Gf\in L_{+}(E,\meas)$ satisfying $S_{N}f\uparrow Gf$
$\meas$-a.e. It is immediate that $Gf_{n}\uparrow Gf$ $\meas$-a.e.\ for
any $\{f_{n}\}_{n\in\mathbb{N}}\subset L_{+}(E,\meas)$
with $f_{n}\uparrow f$ $\meas$-a.e.
Since, on $L^{2}(E,\meas)$, $\{G_{\alpha}\}_{\alpha\in(0,\infty)}$ is the
Laplace transform of $\{T_{t}\}_{t\in(0,\infty)}$, we see that
$S_{t_{n}}f\uparrow Gf$ $\meas$-a.e.\ and $G_{\alpha_{n}}f\uparrow Gf$
$\meas$-a.e.\ for any
$\{t_{n}\}_{n\in\mathbb{N}},\{\alpha_{n}\}_{n\in\mathbb{N}}\subset(0,\infty)$
with $t_{n}\uparrow\infty$, $\alpha_{n}\downarrow 0$.
Moreover, since
$S_{t+N}f=S_{t}f+T_{t}S_{N}f\geq T_{t}S_{N}f$ $\meas$-a.e.\ for
$t\in(0,\infty)$ and $N\in\mathbb{N}$, by letting $N\to\infty$ we have
$T_{t}Gf\leq Gf$ $\meas$-a.e., that is, $Gf$ is $\mathcal{E}$-excessive.
We call this operator $G:L_{+}(E,\meas)\to L_{+}(E,\meas)$ the
\emph{$0$-resolvent associated with $(\mathcal{E},\mathcal{F})$}.
\begin{dfn}[Transience and Recurrence]\label{dfn:rec-tra}
\textup{(1)} $(\mathcal{E},\mathcal{F})$ is called \emph{transient} if and only if
$Gf<\infty$ $\meas$-a.e.\ for some $f\in L_{+}(E,\meas)$
with $f>0$ $\meas$-a.e.

\noindent
\textup{(2)} $(\mathcal{E},\mathcal{F})$ is called \emph{recurrent} if and only if
$\meas(\{0<Gf<\infty\})=0$ for any $f\in L_{+}(E,\meas)$.
\end{dfn}
By \cite[Lemma 1.5.1]{FOT}, $(\mathcal{E},\mathcal{F})$ is transient if and only if
$Gf<\infty$ $\meas$-a.e.\ for any $f\in L^{1}_{+}(E,\meas)$. On the other hand,
by \cite[Theorem 1.6.3]{FOT}, $(\mathcal{E},\mathcal{F})$ is recurrent
if and only if $\ind{}\in\mathcal{F}_{e}$ and $\mathcal{E}(\ind{},\ind{})=0$.

The following proposition is the second step of the proof of Theorem \ref{thm:Liouville}.
\begin{prop}\label{prop:prf-step2}
Assume that $(\mathcal{E},\mathcal{F})$ is recurrent. If
$u\in L^{0}_{+}(E,\meas)$ is $\mathcal{E}$-excessive then
$u\in\mathcal{F}_{e}$ and $\mathcal{E}(u,u)=0$.
\end{prop}
\begin{proof}
Let $n\in\mathbb{N}$. Then $u\wedge n\leq n\ind{}$ $\meas$-a.e.,
$n\ind{}\in\mathcal{F}_{e}$ and $\mathcal{E}(n\ind{},n\ind{})=0$
by the recurrence of $(\mathcal{E},\mathcal{F})$, and
$u\wedge n$ is $\mathcal{E}$-excessive
since so are $u$ and $\ind{}$. Thus
$u\wedge n\in\mathcal{F}_{e}$ and $\mathcal{E}(u\wedge n,u\wedge n)=0$
by Corollary \ref{cor:excessive}. Lemma \ref{lem:Fe-norm}-(2) implies that
$\lim_{n\to\infty}\|v-u\wedge n\|_{\mathcal{F}_{e}}=0$ for some $v\in\mathcal{F}_{e}$
with $\|\cdot\|_{\mathcal{F}_{e}}$ as defined there, and then we easily have
$u=v\in\mathcal{F}_{e}$ and $\mathcal{E}(u,u)=0$.
\end{proof}
As the third step, now we finish the proof of Theorem \ref{thm:Liouville}.
\begin{proof}[Proof of Theorem \textup{\ref{thm:Liouville}}]
$\textrm{\bfseries\upshape 1)}\Rightarrow\textrm{\bfseries\upshape 2)}$
follows by Proposition \ref{prop:prf-step1}, and so does
$\textrm{\bfseries\upshape 1)}\Rightarrow\textrm{\bfseries\upshape 5)}$
by Propositions \ref{prop:prf-step1} and \ref{prop:prf-step2}.
$\textrm{\bfseries\upshape 2)}\Rightarrow\textrm{\bfseries\upshape 3)}$,
$\textrm{\bfseries\upshape 4)}\Rightarrow\textrm{\bfseries\upshape 6)}$ and
$\textrm{\bfseries\upshape 5)}\Rightarrow\textrm{\bfseries\upshape 6)}$ are trivial.

\vskip 4pt\noindent
$\textrm{\bfseries\upshape 1)}\Rightarrow\textrm{\bfseries\upshape 4)}$:
Let $u\in\mathcal{F}_{e}$ be $\mathcal{E}$-excessive in the wide sense,
$n\in\mathbb{N}$ and $u_{n}:=u\wedge n$. Then $u_{n}\in\mathcal{F}_{e}$,
$u_{n}$ is also $\mathcal{E}$-excessive in the wide sense,
$n\ind{}-u_{n}\in\mathcal{F}_{e}\cap L_{+}(E,\meas)$ and hence
$\mathcal{E}(u_{n},u_{n})=\mathcal{E}(u_{n},u_{n}-n\ind{})\leq 0$
by Proposition \ref{prop:Fe-Tt}-(2).
As in the proof of Proposition \ref{prop:prf-step2},
letting $n\to\infty$ we get $\mathcal{E}(u,u)=0$ by Lemma \ref{lem:Fe-norm}-(2),
and hence $u\in\mathbb{R}\ind{}$ by Proposition \ref{prop:prf-step1}.

\vskip 4pt\noindent
$\textrm{\bfseries\upshape 3)}\Rightarrow\textrm{\bfseries\upshape 1)}$:
$(\mathcal{E},\mathcal{F})$ is recurrent since $\ind{}\in\mathcal{F}_{e}$ and
$\mathcal{E}(\ind{},\ind{})=0$.
Let $A\in\mathcal{B}$ be $\mathcal{E}$-invariant. Then
$\ind{A}=\ind{A}\ind{}\in\mathcal{F}_{e}\cap L^{\infty}_{+}(E,\meas)$ and
$0\leq\mathcal{E}(\ind{A},\ind{A})\leq\mathcal{E}(\ind{},\ind{})=0$
by \cite[Theorem 1.6.1]{FOT}. Now $\textrm{\bfseries\upshape 3)}$
implies $\ind{A}\in\mathbb{R}\ind{}$, and hence
either $\meas(A)=0$ or $\meas(E\setminus A)=0$.

\vskip 4pt\noindent
\emph{$\textrm{\bfseries\upshape 6)}\Rightarrow\textrm{\bfseries\upshape 3)}$
when $(E,\mathcal{B},\meas)$ is non-trivial}:
Choose $g\in L^{1}(E,\meas)$ so that $g>0$ $\meas$-a.e.,
and set $E_{c}:=\{Gg=\infty\}$. Then
$\ind{E_{c}}\in\mathcal{F}_{e}\cap L^{\infty}_{+}(E,\meas)$
and $\mathcal{E}(\ind{E_{c}},\ind{E_{c}})=0$ by \cite[Corollary 1.6.2]{FOT},
and $\textrm{\bfseries\upshape 6)}$ together with
Proposition \ref{prop:Fe-Tt}-(3) implies $\ind{E_{c}}\in\mathbb{R}\ind{}$,
i.e., either $\meas(E_{c})=0$ or $\meas(E\setminus E_{c})=0$.
In view of $\textrm{\bfseries\upshape 6)}$ and Proposition \ref{prop:Fe-Tt}-(3),
it suffices to show $\meas(E\setminus E_{c})=0$.

Suppose $\meas(E_{c})=0$, so that $(\mathcal{E},\mathcal{F})$ is transient,
and set $\eta:=g/(1\vee Gg)$. Then $0<\eta\leq g$ $\meas$-a.e.\ and
$\langle\eta,G\eta\rangle\leq\langle g/(1\vee Gg),Gg\rangle\leq\|g\|_{1}<\infty$.
Let $f\in L^{1}_{+}(E,\meas)\cap L^{2}(E,\meas)$ and set $f_{n}:=f\wedge(n\eta)$
for $n\in\mathbb{N}$.
Then $f_{n}\in L^{2}_{+}(E,\meas)$, $Gf_{n}\leq nG\eta<\infty$ $\meas$-a.e.,
$\langle f_{n},Gf_{n}\rangle<\infty$ and $f_{n}\uparrow f$ $\meas$-a.e. Since
$\mathcal{E}(G_{\alpha}f_{n},G_{\alpha}f_{n})
	\leq\langle f_{n},G_{\alpha}f_{n}\rangle\leq\langle f_{n},Gf_{n}\rangle<\infty$
for $\alpha\in(0,\infty)$, Proposition \ref{prop:Fe-Fatou} implies
$Gf_{n}\in\mathcal{F}_{e}$. Since $Gf_{n}$ is $\mathcal{E}$-excessive, so is
$n\wedge Gf_{n}\in\mathcal{F}_{e}\cap L^{\infty}_{+}(E,\meas)$
and $\textrm{\bfseries\upshape 6)}$ yields $n\wedge Gf_{n}\in\mathbb{R}\ind{}$.
Letting $n\to\infty$ and noting $Gf<\infty$ $\meas$-a.e.\ by the transience of
$(\mathcal{E},\mathcal{F})$, we get $Gf\in\mathbb{R}\ind{}$.
Let $\alpha\in(0,\infty)$. Then
$G_{\alpha}f\in L^{1}_{+}(E,\meas)\cap L^{2}(E,\meas)$ and
hence $GG_{\alpha}f\in\mathbb{R}\ind{}$. Letting $n\to\infty$ in
$G_{\alpha}f=G_{1/n}f-(\alpha-1/n)G_{1/n}G_{\alpha}f$ implies that
$G_{\alpha}f=Gf-\alpha GG_{\alpha}f\in\mathbb{R}\ind{}$.
Since $\alpha G_{\alpha}f\to f$ in $L^{2}(E,\meas)$ as $\alpha\to\infty$,
we conclude that $L^{1}_{+}(E,\meas)\cap L^{2}(E,\meas)\subset\mathbb{R}\ind{}$,
contradicting the assumption that $(E,\mathcal{B},\meas)$ is non-trivial.
Thus $\meas(E\setminus E_{c})=0$ follows.
\end{proof}
\subsection*{Acknowledgements}
The author would like to express his deepest gratitude toward Professor Masatoshi
Fukushima for fruitful discussions and for having suggested this problem to him in \cite{F1}.
The author would like to thank Professor Masanori Hino for detailed valuable comments
on the proofs in an earlier version of the manuscript; in particular, the proofs of
Propositions \ref{prop:Fe-Tt} and \ref{prop:prf-step2} have been much simplified
by following his suggestion of the use of Lemma \ref{lem:Fe-norm} and Corollary \ref{cor:excessive}.
The author would like to thank also Professor Masayoshi Takeda and
Professor Jun Kigami for valuable comments.

\begin{thebibliography}{99}
\bibitem{BG} Blumenthal, R.\ M., Getoor, R.\ K.,
  \emph{Markov Processes and Potential Theory},
  Academic Press, New York (1968), republished by Dover Publications, Inc.,
  New York (2007).
%\bibitem{BH} Bouleau, N., Hirsch, F.,
%  \emph{Dirichlet Forms and Analysis on Wiener Space},
%  de Gruyter Studies in Math., \textbf{14}, Walter de Gruyter, Berlin (1991).
\bibitem{CF} Chen Z.-Q. Fukushima M.,
  \emph{Symmetric Markov Processes, Time Change, and Boundary Theory},
  London Math.\ Soc.\ Monographs, \textbf{35},
  Princeton University Press, Princeton (2012).
\bibitem{CK:subh} Chen Z.-Q., Kuwae K.,
  On subharmonicity for symmetric Markov processes,
  \emph{J.\ Math.\ Soc.\ Japan}, \textbf{64} (2012), 1181--1209.
%\bibitem{Dav:STDO} E.\ B.\ Davies,
%  \emph{Spectral Theory and Differential Operators},
%  Cambridge Studies in Advanced Math., \textbf{42},
%  Cambridge University Press, 1995.
%\bibitem{Doob} Doob, J.\ L.,
%  \emph{Classical Potential Theory and Its Probabilistic Counterpart},
%  Grundlehren der mathematischen Wissenschaften, \textbf{262},
%  Springer, Berlin (1984), reprinted in Classics in Mathematics (2001).
\bibitem{F:Fe} Fukushima M.,
  On extended Dirichlet spaces and the space of BL functions,
  in \emph{``Potential theory and stochastics in Albac''},
  Theta Ser.\ Adv.\ Math., \textbf{11}, Theta, Bucharest (2009), 101--110.
\bibitem{F1} Fukushima M., personal communication (December 17, 2008).
%\bibitem{F2} M.\ Fukushima, personal communication, March 12, 2009.
\bibitem{FOT} Fukushima M., Oshima Y., Takeda M.,
  \emph{Dirichlet Forms and Symmetric Markov Processes}, 2nd ed.,
  de Gruyter Studies in Math., \textbf{19}, Walter de Gruyter, Berlin (2011).
\bibitem{FT} Fukushima M., Takeda M.,
  \emph{Markov Processes} (in Japanese), Baifukan, Tokyo (2008).
\bibitem{Get:Exc} Getoor, R.\ K.,
  \emph{Excessive Measures}, Birkh\"{a}user, Boston (1990).
\bibitem{Get:LMN80} Getoor, R.\ K.,
  Transience and recurrence of Markov processes,
  in \emph{``S\'{e}minaire de Probabilit\'{e}s XIV 1978/79''},
  Lecture Notes in Math., \textbf{784}, Springer, Berlin (1980), 397--409.
%\bibitem{Gri:Rec} Grigor'yan, A.,
%  Analytic and geometric background of recurrence and non-explosion of
%  the Brownian motion on Riemannian manifolds,
%  \emph{Bull.\ Amer.\ Math.\ Soc.}, \textbf{36} (1999), 135--249.
%\bibitem{MR} Ma Z.-M., R\"{o}ckner, M.,
%  \emph{Introduction to the theory of (non-symmetric) Dirichlet forms},
%  Springer-Verlag, Berlin-Heidelberg-New York (1992).
\bibitem{Oshima:LMN82} \={O}shima Y.,
  Potential of recurrent symmetric Markov processes and its associated Dirichlet spaces,
  in \emph{``Functional analysis in Markov processes (Katata/Kyoto, 1981)''},
  Lecture Notes in Math., \textbf{923}, Springer, Berlin (1982), 260--275.
\bibitem{Ouh:PA96} Ouhabaz, E.\ M.,
  Invariance of Closed Convex Sets and Domination Criteria for Semigroups,
  \emph{Potential Anal.}, \textbf{5} (1996), 611--625.
%\bibitem{Pin} Pinsky, R.\ G.,
%  \emph{Positive Harmonic Functions and Diffusion: An integrated analytic and
%  probabilistic approach}, Cambridge Studies in Advanced Math., \textbf{45},
%  Cambridge University Press, Cambridge (1995).
\bibitem{Sch:extend} Schmuland, M., Extended Dirichlet spaces,
  \emph{C.\ R.\ Math.\ Acad.\ Sci.\ Soc.\ R.\ Can.},
  \textbf{21} (1999), 146--152.
\bibitem{Sch:Fatou} Schmuland, M.,
  Positivity preserving forms have the Fatou property,
  \emph{Potential Anal.}, \textbf{10} (1999), 373--378.
\bibitem{Shigekawa:convex} Shigekawa I.,
  Semigroups preserving a convex set in a Banach space,
  \emph{Kyoto J.\ Math.}, \textbf{51} (2011), 647--672.
%\bibitem{Sturm:DirI} Sturm, K.-T.:
%  Analysis on local Dirichlet spaces: I. Recurrence, conservativeness and
%  $L^{p}$-Liouville properties, \emph{J.\ Reine Angew.\ Math.}, \textbf{456} (1994), 173--196.
\end{thebibliography}
\end{document}